\documentclass[letterpaper, 10pt, conference, final, twocolumn]{ieeeconf}      %
\IEEEoverridecommandlockouts                              %
\usepackage{epsfig} %

\usepackage{amsmath,graphicx,amsfonts,amssymb,epsfig,subfig,mathrsfs,mathtools}
\usepackage{ntheorem}
\usepackage{cuted}
\usepackage{arydshln}
\usepackage{blkarray}

\usepackage{color}
\usepackage{multirow}
\usepackage{multicol}
\usepackage{lipsum}
\usepackage{rotating}
\usepackage{graphicx}%
\usepackage{algorithm}
\usepackage[utf8]{inputenc}
\usepackage{xspace}

\usepackage{algorithmicx}
\usepackage{algpseudocode}
\usepackage{cite}
\usepackage{tikz}
\usepackage{cancel}
\usepackage{textcomp}
 \definecolor{lin}{RGB}{240,0,0}
 \definecolor{paleblue}{RGB}{0,9,255}
\usepackage[colorlinks=true,linkcolor=lin,citecolor=paleblue,pdfa]{hyperref}

\newcommand{\map}[3]{#1: #2 \rightarrow #3}
\newcommand{\setdef}[2]{\{#1 \; | \; #2\}}
\newcommand{\st}{\ensuremath{\operatorname{s.t.}}}
\newcommand{\real}{\ensuremath{\mathbb{R}}}
\newcommand{\prob}{\ensuremath{\mathbb{P}}}
\newcommand{\realnonnegative}{\ensuremath{\mathbb{R}}_{\ge 0}}

\newcommand{\until}[1]{\{1,\dots, #1\}}
\newcommand{\subscr}[2]{#1_{\textup{#2}}}
\newcommand{\supscr}[2]{#1^{\textup{#2}}}
\newcommand{\vect}[1]{\mathbf{#1}}
\newcommand{\vectorones}[1]{\vect{1}_{#1}}
\newcommand{\vectorzeros}[1]{\vect{0}_{#1}}
\newcommand{\Norm}[1]{\|#1\|}
\newcommand{\trans}[1]{{#1}^\top}

\newcommand{\conv}{\ensuremath{\operatorname{conv}}}
\newcommand{\dom}{\ensuremath{\operatorname{dom}}}

\newcommand{\untilinterval}[2]{\{{#1},\dots, {#2}\}}

\newcommand{\ODAA}{\textsc{Online Data Assimilation Algorithm}\xspace}

\DeclareMathOperator*{\argmin}{argmin}

\makeatletter
\newtheoremstyle{breaknote}%
  {\item{\theorem@headerfont
          ##1\ ##2\theorem@separator}\hskip\labelsep\relax}%
  {\item{\theorem@headerfont
          ##1\ ##2\ (##3)\theorem@separator}\hskip\labelsep\relax}
\makeatother
\theoremstyle{breaknote}
\newtheorem{assumption}{Assumption}[section]
\newtheorem{theorem}{Theorem}[section]

\newtheorem{lemma}{Lemma}[section]

\allowdisplaybreaks
\renewcommand{\baselinestretch}{.977}

\title{Online data assimilation in distributionally robust optimization*}

\author{D. Li$^{1}$ and S. Mart{\'\i}nez$^{1}$
\thanks{*This research
was developed with funding from the DARPA (Lagrange) contract
N660011824027. The views, opinions and/or findings expressed are those
of the author and should not be interpreted as representing the
official views or policies of the Department of Defense or the
U.S. Government.}
\thanks{$\textsuperscript{\textcopyright}$ 2018 IEEE.  Personal use of this material is permitted.  Permission from IEEE must be obtained for all other uses, in any current or future media, including reprinting/republishing this material for advertising or promotional purposes, creating new collective works, for resale or redistribution to servers or lists, or reuse of any copyrighted component of this work in other works. DOI: 10.1109/CDC.2018.8619159.}
\thanks{$^{1}$D. Li and S. Mart{\'\i}nez are with the Department of Mechanical and Aerospace Engineering, University of California San Diego, La Jolla, CA 92092, USA
        {\tt\small lidan@ucsd.edu; soniamd@ucsd.edu}}%
}

\begin{document}

\maketitle

\begin{abstract}
  This paper considers a class of real-time decision making problems
  to minimize the expected value of a function that depends on a
  random variable $\xi$ under an unknown distribution \prob. In this
  process, samples of $\xi$ are collected sequentially in real time,
  and the decisions are made, using the real-time data, to guarantee
  out-of-sample performance. We approach this problem in a
  distributionally robust optimization framework and propose a novel
  \ODAA for this purpose. This algorithm guarantees the out-of-sample
  performance in high probability, and gradually improves the quality
  of the data-driven decisions by incorporating the streaming data. We
  show that the \ODAA guarantees convergence under the streaming data,
  and a criteria for termination of the algorithm after certain number
  of data has been collected.
\end{abstract}
\vspace*{-0.25cm}
\section{Introduction}
\vspace*{-0.25cm}
Online data assimilation is of benefit in many applications that
require real-time decision making under uncertainty, such as optimal
target tracking, sequential planning problems, and robust quality
control. In these problems, the uncertainty is often represented by a
multivariate random variable that has an unknown distribution. Among
available methods, distributionally robust optimization (DRO) has
attracted attention due to its capability to handle data with unknown
distributions while providing out-of-sample performance guarantees
with limited uncertainty samples. To quantify the uncertainty and make
decisions that guarantee the performance reliably, one often needs to
gather a large number of samples in advance. Such requirement,
however, is hard to achieve under scenarios where acquiring samples is
expensive, or when real-time decisions must be made. Further, when the
data is collected over time, it remains unclear what the best the
procedure is to assimilate the data in an ongoing optimization
process.  Motivated by this, this work studies how to incorporate
finitely streaming data into a DRO problem, while guaranteeing
out-of-sample performance via the generation of time-varying
certificates.

\textit{Literature Review:} Optimization under uncertainty is a
popular research area, and as such available methods include
stochastic optimization~\cite{AS-DD-AR:14} and robust
optimization~\cite{AB-LEG-AN:09}. Recently, data-driven
distributionally robust optimization has regained popularity thanks to
its out-of-sample performance guarantees, see e.g.~\cite{AC-JC:17-allerton,AC-JC:17-tac,PME-DK:17,
  RG-AJK:16} and references therein. In this setup, one defines a set
of distributions or \textit{ambiguity set}, which contains the true
distribution of the data-generating system with high
probability. Then, the out-of-sample performance of the data-driven
solution is obtained as the worst-case optimization over the ambiguity
set.
An attractive way of designing these sets is to consider a ball in the
space of probability distributions centered at a reference or
most-likely distribution constructed from the available data. In the
space of distributions, the popular distance metric is the Prokhorov
metric, $\phi$-divergence and the
Wasserstein distance~\cite{AC-JC:17-allerton,PME-DK:17}. Here, following the
paper~\cite{AC-JC:17-allerton}, which proposes a distributed optimization
algorithm for multi-agent settings, we use the Wasserstein distance as it leads to a tractable reformulation of DRO problems.
However, available algorithms in~\cite{AC-JC:17-allerton,PME-DK:17} do not consider the
update of the data-driven solution over time, which serves as the
focus of this work. In terms of the algorithm design, our work
connects to various convex optimization methods~\cite{SB-LV:04} such as the Frank-Wolfe (FW) Algorithm (e.g., conditional gradient algorithm), the Subgradient Algorithm, and their variants, see
e.g.~\cite{CH:74,SLJ-MJ:15} and references therein. Our emphasis on the
convergence of the data-driven solution obtained through a sequence of
optimization problems contrasts with typical optimization algorithms
developed for single (non-updated) problems.

\textit{Statement of Contributions:} Our starting point is a
distributionally robust optimization problem formulation setting
of~\cite{PME-DK:17,AC-JC:17-allerton}, where we further consider that the limited
realizations of the multivariate random variable in the problem are
revealed and collected sequentially over time. As the probability
distribution of the random variable is unknown, we aim to find and
update a real-time data-driven solution based on streaming data. To
guarantee the performance of the data-driven solution with certain
reliability, we follow a DRO approach to solve a worst-case
optimization problem that considers all the probability distributions
in ambiguity sets given as a neighborhood of the empirical
distribution under the Wasserstein metric. Our first contribution is
the generation of such performance guarantee for any real-time
data-driven solution. We achieve this by first finding an equivalent
convex optimization problem over a simplex, and then specializing the
algorithm for efficiently generating a performance certificate of the
data-driven solution with a certain reliability requirement. Based on
the fact that the performance guarantee of data-driven solution with
high probability, our second contribution is the design of a scheme to
find an optimal data-driven solution with the best performance
guarantee under the same reliability. As new data is revealed and
collected sequentially, we specialize the proposed scheme to
assimilate the streaming data.  We show that the resulting \ODAA is
provably correct in the sense that the reliability of the
out-of-sample performance guarantee for the generated data-driven
solution converges to 1 as the number of data samples grows to
infinity, and the data-driven solution with certain performance
guarantee is available any time as soon as the algorithm finish
generating the initial certificate. A convergence analysis of the
proposed algorithm is given, under a user-defined optimality
tolerance. We finally illustrate the performance of the proposed
algorithm in simulation.
\vspace*{-0.25cm}
\section{Preliminaries}\label{sec:pre}
\vspace*{-0.25cm}
\textit{Notations:} Let $\real^m$, $\realnonnegative^m$ and
$\real^{m\times d}$ denote respectively the $m$-dimensional Euclidean
space, the $m$-dimensional nonnegative orthant, and the space of $m
\times d$ matrices, respectively. We use the shorthand notations
$\vectorzeros{m}$ for the column vector $\trans{(0,\cdots,0)} \in
\real^m$, $\vectorones{m}$ for the column vector $\trans{(1,\cdots,1)}
\in \real^m$, and $\vect{I}_m \in \real^{m \times m}$ for the identity
matrix. We let $\vect{x} \in \real^m$ denote a column vector with the
dimension $m$ and $\trans{\vect{x}}$ represents its transpose. We say
a vector $\vect{x} \geq 0$, if all its the entries are
nonnegative.
We use subscripts to index vectors and superscripts to indicate the
component of vector, i.e., $\vect{x}_k \in \real^m$ for $k \in
\{1,2,\ldots,n\}$ and $\vect{x}_k:=\trans{(\vect{x}_k^1,\ldots,
  \vect{x}_k^m)}$. We use $\vect{x}^{i:j}$ to denote the column vector
$\trans{(\vect{x}^i,\vect{x}^{i+1},\ldots,\vect{x}^j)} \in
\real^{j-i+1}$ and $(\vect{x};\vect{y}) \in \real^{m+d}$ indicates the
concatenated column vector from $\vect{x} \in \real^m$ and $\vect{y}
\in \real^d$. An $1$-norm of the vector $\vect{x} \in \real^m$ is
denoted by $\Norm{\vect{x}}$. For matrices $A \in \real^{m \times d}$
and $B \in \real^{p \times q}$, we let $A \oplus B$ denote their
direct sum. The shorthand notation $\oplus_{i=1}^m A_i$ represents
$A_1 \oplus \cdots \oplus A_m$. Given a set of points $I$ in
$\real^m$, we let $\conv(I)$ indicate its convex hull. The gradient of
a real-valued function $\map{f}{\real^m}{\real}$
is written as $\nabla_{\vect{x}}f(\vect{x})$. The $\supscr{i}{th}$
component of the gradient vector is denoted by
$\nabla_if(\vect{x})$. We call the function $f$ \textit{proper} on
$\real^m$ if $f(\vect{x}) < + \infty$ for at least one point $\vect{x}
\in \real^m$ and $f(\vect{x}) >- \infty$ for all $\vect{x} \in
\real^m$. We use $\dom f$ to denote the effective domain of the proper
function $f$, i.e., $\dom f:=\setdef{\vect{x} \in \real^m}{f(\vect{x})
  <+ \infty}$. We say a function $\map{F}{\mathcal{X}\times
  \mathcal{Y}}{\real}$ is \textit{convex-concave} on
$\mathcal{X}\times \mathcal{Y}$ if, for any point
$(\tilde{\vect{x}},\tilde{\vect{y}}) \in \mathcal{X}\times
\mathcal{Y}$, $\vect{x} \mapsto F(\vect{x},\tilde{\vect{y}})$ is
convex and $\vect{y} \mapsto F(\tilde{\vect{x}},\vect{y})$ is
concave.

\textit{Numerical Optimization Methods:} There are mainly two types of
Numerical Optimization methods that serve as the main ingredients of
our \ODAA. One type is given by Frank-Wolfe Algorithm (FWA) variants
and the other is the Subgradient Algorithm. For the Subgradient Algorithm please refer
to~\cite{DPB-AN-AO:03a,YN:13}.

\textit{The Frank-Wolfe Algorithm over a unit simplex.} To solve
convex programs over a unit simplex, here we introduce the FWA
Algorithm following~\cite{SLJ-MJ:15,CH:74}. We define the
$m$-dimensional unit simplex as $\Delta_m:=\setdef{\lambda \in
  \real^m}{\trans{\vectorones{m}} \lambda=1, \; \lambda \geq 0}$. Let
${\Lambda}_{m}$ be the set of all extreme points for the simplex
$\Delta_m$. Consider the minimization of a convex function
$f(\vect{x})$ over $ \Delta_m$; we refer to this problem by $(*)$ and
denote by $\vect{x}^{\star}$ an optimizer of $(*)$. We refer to a
${\vect{x}}^{\epsilon}$ as an $\epsilon$-optimal solution for $(*)$,
if ${\vect{x}}^{\epsilon} \in \Delta_m$ and
$f({\vect{x}}^{\epsilon})-f(\vect{x}^{\star}) \le \epsilon$. We define
a FW search point $\vect{s}^{(k)}$ for the current iteration $k$ at
the feasible point $\vect{x}^{(k)}$, if $\vect{s}^{(k)}$ is an extreme
point such that $\vect{s}^{(k)} \in \argmin_{\vect{x} \in
  \Delta_m} \trans{\nabla
  f(\vect{x}^{(k)})}(\vect{x}-\vect{x}^{(k)})$. With this search point
we define the FW direction at $\vect{x}^{(k)}$ by
$\subscr{d}{FW}^{(k)}:=\vect{s}^{(k)}-\vect{x}^{(k)}$. The classical
Frank-Wolfe Algorithm solves the problem $(*)$ to
${\epsilon}$-optimality by iteratively finding a FW direction and then
solving a line search problem over this direction until an
${\epsilon}$-optimal solution $\vect{x}^{(k)}$ is found, certified by
$\eta^{(k)}:=-{\nabla f(\vect{x}^{(k)})} \cdot \subscr{d}{\small FW}^{(k)}
\leq \epsilon$. Away-step Frank-Wolfe (AFW) Algorithm is an extension of the FWA we used in the following sections, and a linear convergence rate property of the AFW Algorithm is stated in the online version of this paper~\cite{DL-SM:18-extended} for completeness.

\vspace*{-0.25cm}
\section{Problem Description}\label{sec:ProbStat}
\vspace*{-0.25cm}
Consider a decision-making problem given by
\begin{equation}
  \inf\limits_{\vect{x} \in \real^d}{\mathbb{E}_{\prob} [f(\vect{x},\xi)] },
  \label{eq:P}\tag{P}
\end{equation}
where the decision variable $\vect{x}$ on $\real^d$ is to be
determined, the random variable $\map{\xi}{\Omega}{\real^m}$ is
induced by the probability space $(\Omega,\mathcal{F},\prob)$, and the
expectation of $f$ is taken w.r.t.~the unknown distribution
$\prob \in \mathcal{M}(\mathcal{Z})$.
It is not possible to evaluate the objective of~\eqref{eq:P} under
${\vect{x}}$ because $\prob$ is unknown.

This section sets up the framework of an efficient \ODAA that adapts
the decision-making process by using streaming data, i.e.,
\textit{independent and identically distributed} (iid) realizations of
the random variable $\xi$.  Then, we adapt the distributionally
robust optimization approach
following~\cite{AC-JC:17-allerton,PME-DK:17}, to complete the framework.
We omit all proofs in the paper for simplicity and we just report
an outline of the main ideas of the paper. Please see the online version~\cite{DL-SM:18-extended} for more details.

Let $\{\hat{\vect{x}}^{(r)}\}_{r=1}^{\infty}$ be a sequence of
decisions where for each iteration $r$ the decision
$\hat{\vect{x}}^{(r)}$ is feasible for~\eqref{eq:P}. In our \ODAA we
generate $\{\hat{\vect{x}}^{(r)}\}_{r=1}^{\infty}$ while sequentially
collecting iid realizations of the random variable $\xi$ under
$\prob$, denoted by $\hat{\xi}_{n}$, $n=1,2,\ldots$.  This defines a
sequence of streaming data sets, $\hat{\Xi}_{n} \subseteq
\hat{\Xi}_{n+1}$, for each $n$. W.l.o.g. %
we assume that each data set $\hat{\Xi}_{n+1}$ consists of just one
more new data point, i.e., $\hat{\Xi}_{n+1}= \hat{\Xi}_{n} \cup
\{\hat{\xi}_{n+1}\}$ and $\hat{\Xi}_{1}=\{ \hat{\xi}_{1} \}$.
The time between updates of $\hat{\Xi}_n$ and
$\hat{\Xi}_{n+1}$ corresponds to certain time period, which we refer
to as the \textit{$\supscr{n}{th}$ time period}.
The decision sequence obtained during this period is a subsequence of
$\{\hat{\vect{x}}^{(r)}\}_{r=1}^{\infty}$, labeled by
$\{\hat{\vect{x}}^{(r)}\}_{r=r_n}^{r_{n+1}}$. The objective of our
algorithm is to make real-time decisions for~\eqref{eq:P} that have a
potentially low objective value, while adapting the information from
the current data set $\hat{\Xi}_{n}$.

To quantify the quality of the decisions $\{\hat{\vect{x}}^{(r)}\}_{r=1}^{\infty}$,
we introduce the following terms. For each $r$ and the $\supscr{n}{th}$ time period we call
the decision $\hat{\vect{x}}^{(r)} \in \real^d$ \textit{a proper data-driven solution}
of~\eqref{eq:P}, if $\hat{\vect{x}}^{(r)}$ is feasible and its \textit{out-of-sample
  performance}, defined by ${\mathbb{E}_{\prob}
  [f(\hat{\vect{x}}^{(r)},\xi)] }$, satisfies the following \textit{performance guarantee}:
\begin{equation}
  \prob^n({\mathbb{E}_{\prob} [f(\hat{\vect{x}}^{(r)},\xi)] } \leq
  \hat{J}_{n}(\hat{\vect{x}}^{(r)}))\geq 1- \beta_{n} ,
\label{eq:perfgua}
\end{equation}
where the \textit{certificate} $\hat{J}_{n}$ is a function of
$\hat{\vect{x}}^{(r)}$ that indicates the goodness of the performance under
the data set $\hat{\Xi}_{n}$. The \textit{reliability} $(1- \beta_{n})
\in (0,1) \subset \real$ governs the choice of the solution
$\hat{\vect{x}}^{(r)}$ and the resulting certificate
$\hat{J}_{n}(\hat{\vect{x}}^{(r)})$. Finding an approximate
certificate is much easier than finding an exact certificate
$\hat{J}_{n}$ in practice.  Based on this, we call a solution
$\hat{\vect{x}}^{(r)}$ \textit{$\epsilon_1$-proper}, if it
satisfies~\eqref{eq:perfgua} with a approximate certificate,
$\hat{J}_n^{\epsilon_1}$, such that $\hat{J}_{n}(\hat{\vect{x}}^{(r)})
-\hat{J}_{n}^{\epsilon_1}(\hat{\vect{x}}^{(r)}) \leq \epsilon_1$.  The
certificates $\hat{J}_{n}(\hat{\vect{x}}^{(r)})$ and
$\hat{J}_{n}^{\epsilon_1}(\hat{\vect{x}}^{(r)})$, which depend on
$\hat{\vect{x}}^{(r)}$ and the data set $\hat{\Xi}_{n}$, provide an
upper bound to the optimal value of~\eqref{eq:P} with high confidence
$(1-\beta_{n})$ and are to be constructed carefully.

In each time period $n$, given a reliability level $1-\beta_{n}$, our
goal is to approach to an $\epsilon_1$-proper data-driven solution
with a low certificate.  Motivated by this we call any proper
data-driven solution \textit{$\epsilon_2$-optimal}, labeled as
$\hat{\vect{x}}_n^{\epsilon_2}$, if
$\hat{J}_{n}(\hat{\vect{x}}_n^{\epsilon_2}) -\hat{J}_{n}(\vect{x})
\leq \epsilon_2$ for all $\vect{x} \in \real^d$.  Then, for any
$\epsilon_2$-optimal and $\epsilon_1$-proper data-driven solution
$\hat{\vect{x}}_n^{\epsilon_2}$ with certificate
$\hat{J}_{n}^{\epsilon_1}(\hat{\vect{x}}_n^{\epsilon_2})$ and
$\epsilon_1\ll \epsilon_2$, we have the following performance
guarantee:
\begin{equation}
  \prob^n({\mathbb{E}_{\prob} [f(\hat{\vect{x}}_n^{\epsilon_2},\xi)] } \leq
  \hat{J}_{n}^{\epsilon_1}(\hat{\vect{x}}_n^{\epsilon_2}) + \epsilon_1)\geq 1- \beta_{n} .
\label{eq:epsiperfgua}
\end{equation}

We describe now the procedure of the \ODAA to
solve~\eqref{eq:P}. Given tolerance parameters $\epsilon_1$ and
$\epsilon_2$, a sequence of data sets $\{\hat{\Xi}_{n}\}_{n=1}^{N}$
and strictly decreasing confidence levels $\{\beta_{n}\}_{n=1}^{N}$
with $N \rightarrow \infty$ such that $\sum_{n=1}^{\infty} \beta_{n} <
\infty$, the algorithm aims to find a sequence of $\epsilon_2$-optimal
and $\epsilon_1$-proper data-driven solutions,
$\{\hat{\vect{x}}_n^{\epsilon_2}\}_{n=1}^{N}$, associated with the
sequence of the certificates
$\{\hat{J}_n^{\epsilon_1}(\hat{\vect{x}}_n^{\epsilon_2}) \}_{n=1}^{N}$
so that the performance guarantee~\eqref{eq:epsiperfgua} holds for all
$n$. Additionally, as the data streams to infinity, i.e., $n
\rightarrow \infty$ with $N=\infty$, there exists a large enough $n_0$ such that the
algorithm terminates after processing the data set
$\hat{\Xi}_{n_0}$. The algorithm returns a final data-driven solution
$\hat{\vect{x}}_{n_0}^{\epsilon_2}$ such that the performance holds
almost surely, i.e., $\prob^{\infty}({\mathbb{E}_{\prob}
  [f(\hat{\vect{x}}_{n_0}^{\epsilon_2},\xi)] } \leq
\hat{J}_{n_0}^{\epsilon_1}(\hat{\vect{x}}_{n_0}^{\epsilon_2}) +
\epsilon_1)=1$, and meanwhile guarantees the quality  of the certificate
$\hat{J}_{n_0}^{\epsilon_1}(\hat{\vect{x}}_{n_0}^{\epsilon_2})$ to be
close to the optimal objective value of the Problem~\eqref{eq:P}.

To achieve this, consider that
the data set $\hat{\Xi}_{n}$ has been received. We then start by
cheaply constructing a sequence of data-driven solutions
$\hat{\vect{x}}^{(r)}$ with $r \ge r_{n}$, based on the data set
$\hat{\Xi}_{n}$. After a finite number of iterations, if no new data
has been received, the algorithm reaches $r=r_{n+1}$ such that
$\hat{\vect{x}}^{(r_{n+1})} = \hat{\vect{x}}_n^{\epsilon_2}$ is
$\epsilon_2$-optimal, i.e.,
${J}_n^{\epsilon_1}(\hat{\vect{x}}^{(r_{n+1})})$ is $(\epsilon_1
+\epsilon_2)$-close to
${J}_n^{\star}:=\hat{J}_{n}(\hat{\vect{x}}_n^{\star})$ with
$\hat{\vect{x}}_n^{\star} \in \argmin_{x}{\hat{J}_{n}(x)}$. After a
new data point is received, the algorithm finds the next
$\epsilon_2$-optimal data-driven solution
$\hat{\vect{x}}_{n+1}^{\epsilon_2}$ and its certificate
$\hat{J}_{n+1}^{\epsilon_1}(\hat{\vect{x}}_{n+1}^{\epsilon_2})$ with
higher reliability $1 - \beta_{n+1}$. In this way, online
data can be assimilated over time while refining the constructed
$\epsilon_2$-optimal data-driven solutions
$\{\hat{\vect{x}}_n^{\epsilon_2}\}_{n=1}^{\infty}$ with corresponding
certificates
$\{\hat{J}_n^{\epsilon_1}(\hat{\vect{x}}_n^{\epsilon_2})\}_{n=1}^{\infty}$
that guarantee performance with high confidence
$\{1-\beta_{n}\}_{n=1}^{\infty}$.

When the algorithm receives new data set $\hat{\Xi}_{n+1}$ before
reaching to $r=r_{n+1}$, it safely starts from the current data-driven solution
$\hat{\vect{x}}^{(r)}$. The algorithm then
proceeds similarly on the data set $\hat{\Xi}_{n+1}$ by updating the
subsequence index $r_{n+1}$ to the current $r$.

Next, we focus on how to design the certificates based on the
following assumption for $f$:
{
\begin{assumption}[Convexity-concavity %
  and coercivity] %
  The known proper function $\map{f}{\real^d \times \real^m}{\real}$,
  ${(\vect{x},\xi)}\mapsto{f(\vect{x},\xi)}$ is continuously
  differentiable, convex in $\vect{x}$, concave in $\xi$ and
  $f(\vect{x},\tilde{\xi}) \rightarrow +\infty$ as $\Norm{\vect{x}}
  \rightarrow +\infty$ for all $\tilde{\xi} \in
  \real^m$. \label{assump:1}
\end{assumption}
}

\textit{Certificate design via DRO theory:} To design a certificate
$\hat{J}_{n}(\hat{\vect{x}})$ for a given  data-driven
solution $\hat{\vect{x}}$, one can  first use the data set $\hat{\Xi}_{n}$
from $\prob$ to estimate an empirical distribution, $\hat{\prob}^n$,
and let
${\mathbb{E}_{\mathbb{\hat{P}^{\it n}}} [f(\hat{\vect{x}},\xi)] }$ be
the candidate certificate for the performance
guarantee~\eqref{eq:perfgua}. However, such certificate only results in an approximation of
the out-of-sample performance if $\prob$ is unknown
and~\eqref{eq:perfgua} cannot be guaranteed in
probability. Following~\cite{AC-JC:17-allerton,PME-DK:17}, we are to determine an
\textit{ambiguity set} $\mathcal{\hat{P}}_n$ containing all the
possible probability distributions supported on $\mathcal{Z} \subseteq
\real^m$ that can generate $\hat{\Xi}_{n}$ with high confidence. Then
with the given feasible solution $\hat{\vect{x}}$, it is plausible to
consider the worst-case expectation of the out-of-sample performance
for all distributions contained in $\mathcal{\hat{P}}_n$. Such
worst-case distribution offers an upper bound for the out-of-sample
performance with high probability.

Denote by $\subscr{\mathcal{M}}{lt}(\mathcal{Z}) \subset
\mathcal{M}(\mathcal{Z})$ the set of light-tailed probability measures
in $\mathcal{M}(\mathcal{Z})$, we make following assumption:
\begin{assumption}[Light tailed unknown distributions]
  It is assumed that $\prob \in
  \subscr{\mathcal{M}}{lt}(\mathcal{Z})$, i.e., there exists an
  exponent $a>1$ such that: $b:= \mathbb{E}_{\prob}
  [\exp(\Norm{\xi}^a)] < \infty$. \label{assump:2}
\end{assumption}

Assumption~\ref{assump:2} validates the modern measure concentration
result~{\cite[Theorem~2]{NF-AG:15}} on
$\subscr{\mathcal{M}}{lt}(\mathcal{Z})$, which provides an intuition
for considering the Wasserstein ball
$\mathbb{B}_{\epsilon}(\hat{\prob}^n)$ of center $\hat{\prob}^n$ and radius $\epsilon$ as the
ambiguity set $\mathcal{\hat{P}}_n$. Then equipped with the
Wassenstein ball, we are able to provide the certificate that ensures
the performance guarantee in~\eqref{eq:perfgua} for any sequence of
data-driven solutions $\{\hat{\vect{x}}^{(r)}\}_{r=1}^{\infty}$, by $\hat{J}_{n}(\hat{\vect{x}}^{(r)}):= \sup_{\mathbb{Q}  \in \mathcal{\hat{P}}_n} {\mathbb{E}_{\mathbb{Q}} [f(\hat{\vect{x}}^{(r)},\xi)] }$.

\textit{Worst-case distribution reformulation:} To get the certificate
$\hat{J}_{n}(\hat{\vect{x}}^{(r)})$, one needs to solve an infinite dimensional
optimization problem, which is generally hard. Luckily, with an
extended version of the strong duality results for moment
problem~\cite[Lemma~3.4]{AS-MG-MAL:01}, we can reformulate the optimization problem for $\hat{J}_{n}(\hat{\vect{x}}^{(r)})$ into a finite-dimensional convex programming problem:
\begin{equation}\small
\begin{aligned}
  \hat{J}_{n}(\hat{\vect{x}}^{(r)}):=
  \max\limits_{\vect{y}_1,\ldots,\vect{y}_n \in \real^m}&
  \frac{1}{n}\sum_{k=1}^{n}
  f(\hat{\vect{x}}^{(r)},\hat{\xi}_{k}-\vect{y}_k) ,\\
  \st \quad &\frac{1}{n} \sum_{k=1}^{n} \Norm{\vect{y}_k} \le
  \epsilon(\beta_{n}),
\end{aligned}
\label{eq:convJn}\tag{\small P1$_{n}^{(r)}$}
\end{equation}
where $\epsilon(\beta_{n})$ is the radius of of
$\mathbb{B}_{\epsilon(\beta_{n})}$ as calculated in~\cite{DL-SM:18-extended}.
Given an $\epsilon_1$-optimal solution $(\vect{y}^{\epsilon_1}_1,\ldots,\vect{y}^{\epsilon_1}_{n})$ of~\eqref{eq:convJn}, we denote a finite atomic probability measure at
$\hat{\vect{x}}^{(r)}$ in $\mathbb{B}_{\epsilon(\beta_{n})}$ by $\mathbb{Q}_{n}^{\epsilon_1}(\hat{\vect{x}}^{(r)}):=\frac{1}{n}\sum_{k=1}^{n} \delta_{\{\hat{\xi}_{k}-\vect{y}^{\epsilon_1}_k \}}$. Then,
$\mathbb{Q}_{n}^{\epsilon_1}$ is a worst-case distribution that can generate the
data set $\hat{\Xi}_{n}$ with high probability (no less than
$(1-\beta_{n})$).

The concavity requirement in Assumption~\ref{assump:1} ensures that~\eqref{eq:convJn} is a convex problem. Failure of Assumption~\ref{assump:1} may require us to find a relaxed problem of~\eqref{eq:convJn} in order for efficiently generating $\hat{J}_{n}(\hat{\vect{x}}^{(r)})$ in the next section.

\vspace*{-0.25cm}
\section{Certificate Generation} \label{sec:CG}
\vspace*{-0.25cm}
Given the tolerance $\epsilon_1$ and any
feasible solution $\hat{\vect{x}}^{(r)}$,
we present in this section the Certificate Generation Algorithm (CGA) for
efficiently obtaining $\hat{J}_n^{\epsilon_1}(\hat{\vect{x}}^{(r)})$
and the $\epsilon_1$ worst-case distribution,
$\mathbb{Q}_{n}^{\epsilon_1}(\hat{\vect{x}}^{(r)})$ of an
$\epsilon_1$-proper data-driven solution $\hat{\vect{x}}^{(r)}$ over
time, under the sequence of the data sets
$\{\hat{\Xi}_{n}\}_{n=1}^{N}$. To achieve this, we first reformulate
Problem~\eqref{eq:convJn} to a convex optimization problem over a
simplex. Then, we design the CGA to solve
the customized problem to an $\epsilon_1$-optimal solution
efficiently.

For online implementation we have the following
assumption on the computation of the gradient of the function~$f$:

\begin{assumption}[Cheap access of the gradients]
  For any $\vect{x} \in \real^d$, the gradient of the function
  {$\map{h^{\vect{x}}}{\real^m}{\real}$} for
  $h^{\vect{x}}(\vect{y}):=f(\vect{x},\vect{y})$ can be accessed
  cheaply.\label{assump:chgrad}
\end{assumption}

In the $\supscr{n}{th}$ time period with the data set $\hat{\Xi}_{n}$,
we consider the following convex optimization problem over
$\Delta_{2mn}$:
\begin{equation}
\begin{aligned}
 \max\limits_{\vect{v} \in \real^{2mn} }& \frac{1}{n}\sum_{k=1}^{n}
h_k^r((A_n\vect{v})^{(k-1)m+1:km}) ,\\
\st \quad &  \vect{v} \in \Delta_{2mn},
\end{aligned}
\label{eq:JoverSimplex}\tag{\small P2$_{n}^{(r)}$}
\end{equation}
where for each $k
\in \until{n}$, $\hat{\xi}_{k} \in \hat{\Xi}_{n}$ and
$\hat{\vect{x}}^{(r)} \in \real^d$, we define
$\map{h^r_k}{\real^m}{\real}$ as
\begin{center}
  {$h_k^r(\vect{y}):=f(\hat{\vect{x}}^{(r)},\hat{\xi}_{k}-\vect{y})$,} \end{center}
and the matrix $A_n:=[\oplus_{i=1}^{n}I_m, -\oplus_{i=1}^{n}I_m] \in
\real^{mn \times 2mn}$ where the first $mn$ columns of $A_n$
constitute the natural basis for the space $\real^{mn}$. The simplex
is defined by $\Delta_{2mn}:=\setdef{\vect{v} \in
  \real^{mn}}{\trans{\vectorones{2mn}}\vect{v}=n\epsilon(\beta_{n})
  ,\; \vect{v} \geq 0}$ and we denote by ${\Lambda}_{2mn}$ the set of
all the extreme points for the simplex $\Delta_{2mn}$.

One can prove that solving~\eqref{eq:convJn} is equivalent to solving~\eqref{eq:JoverSimplex} for any feasible solution $\hat{\vect{x}}^{(r)}$ in every time period $n$, we refer to the online version~\cite{DL-SM:18-extended} for details.

The Frank-Wolfe Algorithm variants, such as the Simplicial
Algorithm~\cite{CH:74} and the AFW algorithm~\cite{SLJ-MJ:15}, are
known to be well suited for problems of the
form~\eqref{eq:JoverSimplex}.
The advantage of these is that they can handle the constraints of
Problem~\eqref{eq:JoverSimplex} via linear programming subproblems
(LP) that result from the way in which the FW search point is found
in~\cite[Section~2]{DL-SM:18-extended}. Intuitively, the following is
done. For a number of iterations~$l$, the following problems are
solved alternatively:
\begin{equation}
\begin{aligned}
  \small \max_{\vect{v} \in \real^{2mn} }& \frac{1}{n}\sum_{k=1}^{n}
  \left\langle  \nabla h_k^r(\vect{y}^{(l-1)}_k), \cdots \right.\\ & \left.\hspace*{0.3cm} (A_n\vect{v})^{(k-1)m+1:km}-\vect{y}^{(l-1)}_k\right\rangle ,\\
  \st \quad & \vect{v} \in \Delta_{2mn},
\label{eq:LP}
\end{aligned}\tag{LP$^{(l)}$}
\end{equation}
\vspace*{-0.4cm}
\begin{equation}
\begin{aligned}
 \max\limits_{\gamma \in \real^{T+1} }& \frac{1}{n}\sum_{k=1}^{n}
h_k^r(\sum_{i=0}^{T}\gamma^i  \tilde{\vect{y}}_k^{(i)}) ,\\
\st \quad & \gamma \in \Delta_{T}.
\label{eq:CP}
\end{aligned}\tag{CP$^{(l)}$}
\end{equation}
Notice that the search points generated for the linear
subproblem~\eqref{eq:LP} at iteration $l$ are the extreme points of
the feasible set $\Delta_{2mn}$. We denote by $I_{n}^{(l)}$ the set of
these points. Considering the convex hull of $I_{n}^{(l)}$,
parametrized by the convex combination coefficients $\gamma$ of the
points in $I_{n}^{(l)}$, an implicit feasible set $\conv(I_{n}^{(l)})$
in a lower dimensional space can be constructed. Motivated by this,
our Certificate Generation Algorithm iteratively solves the linear
subproblem~\eqref{eq:LP}, enlarges the implicit feasible set
$\conv(I_{n}^{(l)})$, and then searches a maximizer of the objective
function of~\eqref{eq:JoverSimplex} over $\conv(I_{n}^{(l)})$
(represented as $\Delta_T$ in subproblem~\eqref{eq:CP}). This process
is repeated to the next iteration $l+1$, and follows until an
$\epsilon_1$-optimal solution is found.  Later we will see that the
set $I_{n}^{(l)}$ allows to generate the certificate when assimilating
data. We call this set the \emph{candidate vertex set.}

For the above problems, notice that the subproblem~\eqref{eq:LP}
maximizes a linear function over a simplex, therefore it is
computationally cheap and an optimizer $\vect{v}^{(l)}$ is
equivalently computed by choosing a sparse vector with only one
positive entry, i.e., an extreme point of the feasible set
of~\eqref{eq:LP}, such that the nonzero component of $\vect{v}^{(l)}$
has the largest weight in the linear cost function of
Problem~\eqref{eq:LP}.

We refer to the online version of this paper~\cite{DL-SM:18-extended}
for complete description of the Certificate Generation Algorithm
(denoted by~\cite[Algorithm~4]{DL-SM:18-extended}) and its finite convergence to achieve
$\hat{J}_n^{\epsilon_1}(\hat{\vect{x}}^{(r)})$ and
$\mathbb{Q}_{n}^{\epsilon_1}(\hat{\vect{x}}^{(r)})$ for any
data-driven solution $\hat{\vect{x}}^{(r)}$ in the $\supscr{n}{th}$
time period with the data set $\hat{\Xi}_{n}$.

The worst-case computational bound of the
Certificate Generation Algorithm at the iteration $l+1$, associated
with the candidate solution
$(\vect{y}^{(l)}_1,\ldots,\vect{y}^{(l)}_n)$, is
 $ \hat{J}_n(\hat{\vect{x}}^{(r)}) -\hat{J}_n^{\eta^{(l+1)}}(\hat{\vect{x}}^{(r)}) \le 2mn\kappa^l \rho$,
where $\kappa:=1-\frac{\mu_f}{4C_f} \in (0,1) \subset \real$ is
related to local strong convexity of $f$ over $\Delta_{2mn}$, and
$\rho:=\hat{J}_n(\hat{\vect{x}}^{(r)})-\hat{J}_n^{\eta^{(1)}}(\hat{\vect{x}}^{(r)})
\leq \eta^{(1)}$ quantifies the initial distance of the objective
function $\hat{J}_n$ and $\hat{J}_n^{\eta^{(1)}}$ at
$\hat{\vect{x}}^{(r)}$.  In other words, given the tolerance
$\epsilon_1$, in the worst case we need at least $l \geq \phi(n):=
{\log}_{\kappa}(\frac{\epsilon_1}{2mn\rho})$ computational steps to
find $\hat{J}_n^{\epsilon_1}(\hat{\vect{x}}^{(r)})$.

However, how to generate
$\hat{J}_n^{\epsilon_1}(\hat{\vect{x}}^{(r)})$ and
$\mathbb{Q}_{n}^{\epsilon_1}(\hat{\vect{x}}^{(r)})$ online is unclear
for each data-driven solution $\hat{\vect{x}}^{(r)}$. This is because
that as the time period $n$ moves, we need to not only obtain
$\hat{J}_n^{\epsilon_1}(\hat{\vect{x}}^{(r)})$ and
$\mathbb{Q}_{n}^{\epsilon_1}(\hat{\vect{x}}^{(r)})$ sufficiently fast,
but also finding them by solving the Problem~\eqref{eq:JoverSimplex}
under a different data set $\hat{\Xi}_{n}$. As the size of
$\hat{\Xi}_{n}$ grows, the dimension of the
Problem~\eqref{eq:JoverSimplex} increases. To deal these challenges,
our Certificate Generation Algorithm exploits the relationships among
the Problems~\eqref{eq:JoverSimplex} with different data set
$\hat{\Xi}_{n}$ by adapting the candidate vertex set
$I_{n}^{(l)}$.

When the average data streaming rate is slower than the computational
bound $\phi(1)$, we claim that Certificate Generation Algorithm can always
find the certificate for each data set $\hat{\Xi}_{n}$. This is
because in each time period $n$ on average, we only have $2mn$ extreme
points, and $2m(n-1)$ has been explored due to the adaptation of the
candidate vertex set $I_{n}^{(0)}$. This indicates that in the
worst-case situation the average data streaming rate should be lower
than this value, in order to efficiently update the certificate for
the sequence of the data-driven solutions.
\vspace*{-0.25cm}
\section{An $\epsilon_2$-optimal performance guarantee }
\label{sec:LowJ}
\vspace*{-0.25cm}
In this section, we approach the construction of a sequence of the
$\epsilon_2$-optimal data-driven solutions
$\{\hat{\vect{x}}_n^{\epsilon_2}\}_{n=1}^{\infty}$, associated with
$\epsilon_2$-lowest certificates
$\{\hat{J}_n^{\epsilon_1}(\hat{\vect{x}}_n^{\epsilon_2})\}_{n=1}^{\infty}$
over time, under the sequence of the data sets
$\{\hat{\Xi}_{n}\}_{n=1}^{\infty}$. Specifically in the
$\supscr{n}{th}$ time period, we start from
$\hat{\vect{x}}^{(r)}:=\hat{\vect{x}}^{(r_n)}$ with its associated
$\epsilon_1$-optimal certificate
$\hat{J}_n^{\epsilon_1}(\hat{\vect{x}}^{(r)})$, and as the iteration
$r$ grows, we are to find a sequence of $\hat{\epsilon}_1$-proper
data-driven solutions, $\{\hat{\vect{x}}^{(r)}\}_{r=r_n}^{r_{n+1}}$,
which converge to $\hat{\vect{x}}^{\epsilon_2}_{n}$ quickly and result in $\hat{J}_n^{\epsilon_1}(\hat{\vect{x}}_n^{\epsilon_2})$. We use a Subgradient Algorithm to obtain $\hat{\vect{x}}_{n}^{\epsilon_2}$, via a valid $\epsilon_1$-subgradient of the certificate function $\hat{J}_n^{\epsilon_1}(\hat{\vect{x}}^{(r)})$, which denoted by $g_n^{\epsilon_1}(\hat{\vect{x}}^{(r)})$ and can be computed as shown in the extended version~\cite{DL-SM:18-extended}.

However, for every time we generate a new data-driven solution
$\hat{\vect{x}}^{(r+1)}$, the $\epsilon_1$-optimal extreme
distribution $\hat{\mathbb{Q}}_{n}^{\epsilon_1}(\hat{\vect{x}}^{(r)})$
associated with the last solution $\hat{\vect{x}}^{(r)}$ may not be a
valid $\epsilon_1$-optimal extreme distribution for
$\hat{\vect{x}}^{(r+1)}$. To reduce the number of computations needed
to obtain the new certificate for $\hat{\vect{x}}^{(r+1)}$, we denote
by ${g}_{n}^{\epsilon^{(r)}}(\hat{\vect{x}}^{(r)})$ the
$\epsilon^{(r)}$-subgradient at $\hat{\vect{x}}^{(r)}$, where $
\epsilon^{(r)}$ may be greater than $ \epsilon_1$ for each $r$. Then
by properly designing a sequence $\{\epsilon^{(r)} \}$, upper bounded
by $\hat{\epsilon}_1$, and estimating the $\epsilon^{(r)}$-optimal
extreme distributions, we will achieve a suboptimal proper data-driven
solution efficiently.

Here, we employ the $\hat{\epsilon}_1$-Subgradient Algorithm with
$\hat{\epsilon}_1\gg\epsilon_1$, the divergent but square-summable
step size rule, and scaled direction as follows: $  \hat{\vect{x}}^{(r+1)}=\hat{\vect{x}}^{(r)}-\alpha^{(r)}
  {\hat{g}_{n}^{\hat{\epsilon}_1}(\hat{\vect{x}}^{(r)})}/{\max\{
    \Norm{ \hat{g}_{n}^{\hat{\epsilon}_1}(\hat{\vect{x}}^{(r)}) } \; ,
    \; 1 \} }.$
The
estimated $\hat{\epsilon}_1$-subgradient
$\hat{g}_{n}^{\hat{\epsilon}_1}(\hat{\vect{x}}^{(r)})$ at each
iteration $r$ is constructed and updated via the following
considerations. Every time we generate the $\epsilon_1$-optimal
certificate from the Certificate Generation Algorithm at iteration $r$, the
estimated $\hat{\epsilon}_1$-subgradient is constructed by
$\hat{\mathbb{Q}}_{n}^{\epsilon_1}(\hat{\vect{x}}^{(r)})$ using lemma for the easy access of the subgradients~\cite{DL-SM:18-extended}, i.e.,
${g}_{n}^{\epsilon_1}(\hat{\vect{x}}^{(r)})
\in \partial_{\hat{\epsilon}_1}
\hat{J}_{n}(\hat{\vect{x}}^{(r)})$. During the execution of the
$\hat{\epsilon}_1$-Subgradient Algorithm, we check for the
$\hat{\epsilon}_1$-optimality of the certificate generated from
$\hat{\mathbb{Q}}_{n}^{\epsilon_1}(\hat{\vect{x}}^{(r)})$ at each
subsequent iteration $\hat{r}$, using~\cite[Algorithm~3]{DL-SM:18-extended}. If the
obtained suboptimality gap is such that $\eta > \hat{\epsilon}_1$ at
$\hat{r} > r$, we generate a new $\epsilon_1$-optimal distribution
$\hat{\mathbb{Q}}_{n}^{\epsilon_1}(\hat{\vect{x}}^{(\hat{r})})$ via
Certificate Generation Algorithm and estimate the
$\hat{\epsilon}_1$-subgradient using
$\hat{\mathbb{Q}}_{n}^{\epsilon_1}(\hat{\vect{x}}^{(\hat{r})})$.
Otherwise, the certificate at $\hat{\vect{x}}^{(\hat{r})}$ is
constructed using
$\hat{\mathbb{Q}}_{n}^{\epsilon_1}(\hat{\vect{x}}^{(r)})$.

From the above construction, we see that each $\epsilon^{(r)}$,
associated with a $\hat{g}_{n}^{\hat{\epsilon}_1}(\hat{\vect{x}}^{(r)})$,
is such that $\epsilon^{(r)} \le \hat{\epsilon}_1$. Then, we have the
following lemma for the convergence of the
$\hat{\epsilon}_1$-Subgradient Algorithm in the $\supscr{n}{th}$ time period.

\begin{lemma}[Convergence of the $\hat{\epsilon}_1$-Subgradient Algorithm to the $\epsilon_2$-optimal solution given $\hat{\Xi}_{n}$]
  In each time period $n$ with an initial data-driven solution
  $\hat{\vect{x}}^{(r_n)}$, assume the subgradients defined by $\hat{g}_{n}^{\epsilon}(\hat{\vect{x}}^{(r)})$
  are uniformly bounded, i.e., there exists a constant $L>0$ such that
  $\Norm{ \hat{g}_{n}^{\epsilon}(\hat{\vect{x}}^{(r)}) } \leq L$ for
  all $r\geq r_{n}$ and $\epsilon \leq \hat{\epsilon}_1$. Let $\mu :=
  \max\{ L , \, 1 \}$. Given a predefined $\epsilon_2 >0$, and let the
  certificate tolerance $\epsilon_1$ and the subgradient bound
  $\hat{\epsilon}_1$ such that $0 < \epsilon_1 \ll \hat{\epsilon}_1 <
  \epsilon_2/ \mu$, then there exists a large enough number $\bar{r}$
  such that the designed $\hat{\epsilon}_1$-Subgradient
  Algorithm has the following performance bounds:
  \begin{equation*} \small
    \min\limits_{k\in \untilinterval{r_{n}}{r}} \{
    \hat{J}_{n}(\hat{\vect{x}}^{(k)})\} -
    \hat{J}_{n}(\hat{\vect{x}}_{n}^{\star}) \leq \epsilon_2, \quad \forall
    \; r \geq  \bar{r},
  \end{equation*}
  and terminates at the iteration $r_{n+1}:=\bar{r}$ with an
  $\epsilon_2$-optimal solution under the data set $\hat{\Xi}_{n}$.
\label{lemma:convergeX}
\end{lemma}

\vspace*{-0.25cm}
\section{Data Assimilation} \label{sec:DataAssimilate}
\vspace*{-0.25cm}
We now describe a full algorithm to assimilate data online. The whole \ODAA
starts from some random initial data-driven solution. Then, for each
given set of data points, we first generate its certificate via
Certificate Generation Algorithm, after which an $\epsilon$-proper
data-driven solution is obtained, then we execute the Subgradient
Algorithm to achieve a lower certificate. During the last set of
iterations, the certificate may be lost and Certificate Generation
Algorithm may have to be rerun again, and resume the Subgradient
Algorithm after obtaining a valid certificate. If no data points come
in, the algorithm terminates as soon as the Subgradient Algorithm
terminates.

When there is streaming data, the algorithm needs to incorporate new
data points every time they become available. Because of this, the
feasible set of the Problem~\eqref{eq:convJn} changes. This affects
the dimension of Problem~\eqref{eq:convJn}, which grows by $m$, and
results into an increase of the dimension of~\eqref{eq:LP} by
$2m$. Second, the reliability increase from $\beta_{n}$ to
$\beta_{n+1}$ results into a smaller radius $\epsilon(\beta_{n+1})$ of
the Wasserstein ball $\mathbb{B}_{\epsilon(\beta_{n+1})}$.

Depending on the stage the new data point comes in, different
strategies for generating initial point that is feasible for the new
optimization problem can be considered. When data comes in during the execution of the $\epsilon$-Subgradient Algorithm at iteration $r$, we use a current best
$\hat{\epsilon}_1$-proper data-driven solution as the initial
data-driven solution for the $\epsilon_2$-optimal data-driven solution
$\hat{\vect{x}}^{\epsilon_2}_{n+1}$, i.e.,
$\hat{\vect{x}}^{(r_{n+1})}:=\supscr{\hat{\vect{x}}}{best}_{n} \in
\argmin_{k\in
  \untilinterval{r_{n}}{r}}\{\hat{J}_{n}(\hat{\vect{x}}^{(j)})\}$.
	The other initial data, such as $(\vect{y}^{(0)}_1,\ldots,\vect{y}^{(0)}_{n})$, $I_{n}^{(0)}$ and $\{\tilde{\vect{y}}_k^{(i)}\}_{i \in I_{n}^{(0)}}$, can be constructed following the same idea as when data point comes in during the execution of Certificate Generation Algorithm, the details of which are in~\cite{DL-SM:18-extended}. By such scheme the online data can be assimilated into sequence of optimization problems~\cite[the Algorithm~5]{DL-SM:18-extended}.

The \ODAA has the anytime property, meaning that the performance
guarantee is provided anytime, as soon as the first
$\epsilon_1$-proper data-driven solution is found. The algorithm
then tries to make decisions that achieve lower certificates
with higher reliability until we achieve the lowest
possible certificate and guarantee the performance almost surely.

The transient behavior of the \ODAA is naturally affected by the data
streaming rate and the rate of convergence of intermediate algorithms
(the assimilation rate). To further describe the effect of the data streaming rate, we call the data set stream $\{\hat{\Xi}_{n}\}_{n=1}^{N}$ \textit{sufficiently
  slow in the $\supscr{n}{th}$ time period}, if we can find an
$\hat{\vect{x}}_n^{\epsilon_2}$ in the $\hat{\epsilon}_1$-Subgradient
Algorithm during the time period $n$.
When the data streaming rate and assimilation rate are the same, and they are sufficiently slow for all $n$, the
\ODAA guarantees to find a low certificate with a good data-driven solution as established by the following finite convergence result.
\begin{theorem}[Finite convergence of the \ODAA ]
  Given any tolerance $\epsilon_1$, $\epsilon_2$, $\epsilon_3>0$ and sufficiently
  slow data streaming sets $\{\hat{\Xi}_{n}\}_{n=1}^{\infty}$. Then, there exists a large
  enough number $n_0(\epsilon_3) >0$, such that the algorithm
  terminates in finite time with a sequence of $\epsilon_2$-optimal
  $\epsilon_1$-proper data-driven solutions
  $\{\hat{\vect{x}}_n^{\epsilon_2}\}_{n=1}^{n_0}$ associated with the
  sequence of the certificates
  $\{\hat{J}_n^{\epsilon_1}(\hat{\vect{x}}_n^{\epsilon_2})
  \}_{n=1}^{n_0}$ so that the performance
  guarantee~\eqref{eq:epsiperfgua} holds for all $n \leq n_0$. Moreover, we have a guaranteed $\epsilon_2$-optimal and
  $\epsilon_1$-proper data-driven solution
  $\hat{\vect{x}}_{n_0}^{\epsilon_2}$ and a certificate
  $\hat{J}_{n_0}^{\epsilon_1}(\hat{\vect{x}}_{n_0}^{\epsilon_2})$ such
  that the performance guarantee holds almost surely, i.e.,
  \begin{equation*}\small
    \begin{aligned}
      \prob^{\infty}({\mathbb{E}_{\prob}
        [f(\hat{\vect{x}}_{n_0}^{\epsilon_2},\xi)] } \leq
      \hat{J}_{n_0}^{\epsilon_1}(\hat{\vect{x}}_{n_0}^{\epsilon_2}) +
      \epsilon_1)=1,
    \end{aligned} \label{eq:asperf}
  \end{equation*}
  and meanwhile the quality of the designed certificate
  $\hat{J}_{n_0}^{\epsilon_1}(\hat{\vect{x}}_{n_0}^{\epsilon_2})$ is
  guaranteed, i.e., for all the rest of the data sets
  $\{\hat{\Xi}_{n}\}_{n=n_0}^{\infty}$, any element in the desired
  certificate sequence
  $\{\hat{J}_n^{\epsilon_1}(\hat{\vect{x}}_n^{\epsilon_2})
  \}_{n=n_0}^{\infty}$ satisfies
  \begin{equation*}
    \begin{aligned}
      \sup_{n \geq n_0}
      \hat{J}_{n}^{\epsilon_1}(\hat{\vect{x}}_{n}^{\epsilon_2})
      \leq J^{\star} + \epsilon_1 +\epsilon_2 +\epsilon_3,
    \end{aligned}
		\label{eq:perfbound}
  \end{equation*}
  where $J^{\star}$ is the optimal
  objective value for~\eqref{eq:P}.
\label{thm:assimilation}
\end{theorem}

\vspace*{-0.25cm}
\section{Simulation results}\label{sec:Sim}
\vspace*{-0.25cm}
In this section, we demonstrate the application of the \ODAA to find
an $\epsilon$-proper data-driven solution $x \in \real^{30}$ for
Problem~\eqref{eq:P}. We consider $N=50$ iid sample points $\{
\hat{\xi}_{k}\}_{k=1}^{N}$ streaming randomly in between every $1$ to
$3$ seconds with each data point $\hat{\xi}_{k} \in \real^{10}$ a
realization of the unknown distribution $\prob$. Here, we assume that
the unknown distribution is a mixture of the multivariate uniform
distribution on $[-2,2]^{10}$ and the multivariate normal distribution
$\mathcal{N}(2.5 \cdot\vectorones{10},4 \cdot I_{10})$. We assume the
cost function $\map{f}{\real^{30} \times \real^{10}}{\real}$ to be
$f(\vect{x},\xi):=
\trans{\vect{x}}A\vect{x}+\trans{\vect{x}}B\xi+\trans{\xi}C\xi$ with
random values for the positive semi-definite matrix $A \in \real^{30
  \times 30}$, $B \in \real^{30 \times 10}$ and negative definite
matrix $C \in \real^{10 \times 10}$. Let the reliability
$1-\beta_{n}:=1-0.95e^{1-\sqrt{n}}$ and use the parameter $c_1=2$,
$c_1=1$ to design the radius $\epsilon(\beta_{n})$ of the Wasserstein
ball. We sample the initial data-driven solution
$\hat{\vect{x}}^{(0)}$ from the uniform distribution
$[0,10]^{30}$. The tolerance for the algorithm is $\epsilon_1=10^{-5}
$, $\epsilon_2= 10^{-6}$, and $\epsilon_3=10^{-6}
$.

To evaluate the quality of the obtained $\epsilon$-proper data-driven
solution with the streaming data, we estimate the optimizer
of~\eqref{eq:P}, $\vect{x}^{\star}$, by minimizing the average value
of the cost function $f$ for a validation data set with
$\subscr{N}{val}=10^4$ data points randomly generated from the
distribution $\prob$ (in the simulation case $\prob$ is known). We
take the resulting objective value as the estimated optimal objective
value for Problem~\eqref{eq:P}, i.e.,
$J^{\star}:=\hat{J}^{\star}({\vect{x}}^{\star})$. We calculate
$\hat{J}^{\star}({\vect{x}}^{\star})$ using the underline distribution
$\prob$, serving as the true but unknown scale to evaluate the
goodness of the certificate obtained throughout the algorithm.

Figure~\ref{fig:Jx} shows the evolution of the certificate sequence
$\{\hat{J}_n^{\epsilon_1}(\hat{\vect{x}}^{(r)})
\}_{n=1,r=1}^{N,\infty}$ for the decision sequence
$\{\hat{\vect{x}}^{(r)}\}_{r=1}^{\infty}$. The blue line in the
Figure~\ref{fig:Jx} shows the relative goodness of the certificates
for the currently used $\epsilon_1$-proper data-driven solution
$\hat{\vect{x}}^{(r)}$ calibrated by the estimated optimal value
$J^{\star}$ over time. The red points indicate that a new certificate
$\hat{J}_{n+1}^{\epsilon_1}(\hat{\vect{x}}_{(r)}(t))$ is processing
when the new data set is incorporated, while at these time intervals
the old certificate
$\hat{J}_{n}^{\epsilon_1}(\hat{\vect{x}}_{n}^{\epsilon_2})$,
associated with the $\epsilon_2$-optimal and $\epsilon_1$-proper
data-driven solution $\hat{\vect{x}}_{n}^{\epsilon_2}$, is still valid
to guarantee the performance under the old reliability
$\beta_{n}$. This situation commonly happens when a new data set
$\hat{\Xi}_{n+1}$ is streamed in and a new certificate
$\hat{J}_{n+1}^{\epsilon_1}(\hat{\vect{x}}_{(r)}(t))$ is yet to be
obtained.  It can be seen that after a few samples streamed, the
obtained certificate becomes close to the estimated true optimal value
$J^{\star}$ within the $10\%$ range.
 \begin{figure}[tbp]
 	 \centering
 		\includegraphics[width=0.4\textwidth]{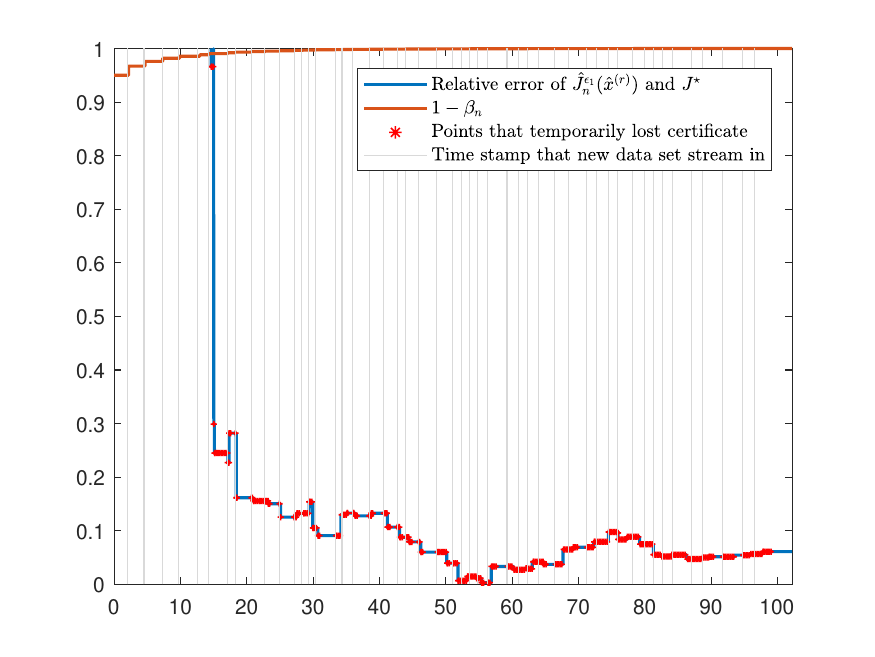}
                \caption{\small Relative goodness of the certificates for the
                  performance guarantee of an $\epsilon_1$-proper
                  data-driven solution. The $x$-axis is time
                  (seconds) and the y-axis
                  plots the relative goodness function
                  $R(t):=|{{(\hat{J}_{n}(\hat{\vect{x}}^{(r)}(t))
                      -J^{\star} )}/{J^{\star}}}|$. }
 	 \label{fig:Jx}
  \end{figure}
\vspace*{-0.25cm}
\section{Conclusions}\label{sec:Conclude}
\vspace*{-0.25cm}
In this paper, we have proposed the \ODAA for real-time data-driven solutions of~\eqref{eq:P} with guaranteed out-of-sample performance. Such solutions are available any time during the execution of the algorithm, and the optimal data-driven solution are approached with a
(sub)linear convergence rate. The algorithm terminates after
collecting sufficient amount of data to make a good decision. Future work will generalize the results for weaker assumptions of the problem and potentially extend the algorithm
to scenarios that include system dynamics.

 \bibliographystyle{IEEEtran}
 \bibliography{../../../../bib/alias,../../../../bib/SMD-add,../../../../bib/SM,../../../../bib/JC,../../../../bib/Main-sonia}
\end{document}